\documentclass[12pt,reqno]{amsart}
\usepackage{tikz-cd}
\usepackage{amscd}
\usepackage{amsmath, amsfonts,amssymb,amsthm,mathrsfs,xypic}
\usepackage{graphicx}
\usepackage{fancybox}
\calclayout
\usepackage{amsmath}
\usepackage{hyperref}
\usepackage[all]{xy}
\usepackage{extarrows}

\usetikzlibrary{arrows} 
\usepackage{mathtools}
\usepackage{mathrsfs,stmaryrd,wasysym}
\usepackage{graphics}
\usepackage{xypic}
\usepackage{multirow}
\usepackage{makecell}

\usepackage[english]{babel}
\usepackage{amsmath}

\allowdisplaybreaks[4]

\usepackage{cases}

\usepackage{dsfont}
\usepackage{color}
\SelectTips{cm}{}
\calclayout

\numberwithin{equation}{section}

\theoremstyle{plain}
\newtheorem{theorem}[equation]{Theorem}
\newtheorem{proposition}[equation]{Proposition}
\newtheorem{lemma}[equation]{Lemma}

\theoremstyle{definition}

\theoremstyle{remark}
\newtheorem{remark}[equation]{Remark}

\def\mcZ{\mathcal{Z}}
\def\mcF{\mathcal{F}}

\def\mcH{\mathcal H}
\def\mcC{\mathcal C}
\def\mcR{\mathcal R}

\def\parti{\partial}

\def\bi{\text{\boldmath$i$}}

\def\bj{\text{\boldmath$j$}}

\begin{document}
	
	\title[The Brundan-Kleshchev isomorphism revisited]{The Brundan-Kleshchev isomorphism revisited}
	\author [\tiny{Fan Kong and Zhi-Wei Li}] {Fan Kong and Zhi-Wei Li*}
	
	\begin{abstract} We give a short and unified proof of the Brundan-Kleshchev isomorphism between blocks of cyclotomic Hecke algebras and cyclotomic Khovanov-Lauda-Rouquier algebras of type A.
	\end{abstract}
	
	\date{\today}
	\thanks{*Corresponding author}
	\thanks{{\em 2000 Mathematics Subject Classification:} 20C08.}
	\thanks{The first author is supported by the NSF of China (No. 11501057) and the Doctoral Fund of Southwest University (SWU-118003); The second author is supported by the NSF of China (No. 11671174).}
	
	\address{Fan Kong\\ School of Mathematics and Statistics \\
		Southwest University \\ Chongqing 400715\\ P. R. China.}
	\email{kongfan85@126.com}

	\address{Zhi-Wei Li\\ School of Mathematics and Statistics \\
		Jiangsu Normal University\\	Xuzhou 221116 Jiangsu \\ P. R. China.}
	\email{zhiweili@jsnu.edu.cn}
	\maketitle
	
	\setcounter{tocdepth}{1}
	
	\section{Introduction}

In 2008, Brundan and Kleshchev \cite{Brundan-Kleshchev09} gave an explicit isomorphism between blocks of cyclotomic affine Hecke algebras of symmetric groups, or the corresponding degenerated cyclotomic affine Hecke algebras, and cyclotomic KLR algebras of type A. On the other hand, in 1989, Lusztig showed that there is a natural isomorphism between affine Hecke algebras and their graded versions \cite{Lusztig89}. Motivated by Lusztig's work, we introduced the semi-rationalization extensions of affine Hecke algebras in \cite{KoLi}, where an isomorphism between direct sums of blocks of cyclotomic affine Hecke algebra of all types and their cyclotomic semi-rationalization algebras was given. 
\vskip5pt
The aim of this paper is to give a shorter proof of Brundan and Kleshchev isomorphism using the machinery of our general result.
\vskip5pt
\noindent{\bf Main result} (Theorem \ref{thm:mainresult})\ {\itshape The cyclotomic KLR algebra of type A and blocks of the cyclotomic $($degenerate$)$ affine Hecke algebra of a symmetric group can be realized as the same quotient of an algebra.}
\vskip5pt
\noindent It can be illustrated by the following diagram of algebras.
\[
\begin{tikzpicture}
\node at(0,0){$\tilde{\mathcal{L}}$};
\node at (-1,0){$\mathcal{L}$};
\node at (1.5,0){$\tilde{\mathcal{L}}(\Lambda)$};
\node at (4,1){$\mcR(\Lambda)$};
\node at (4.5,-1){$\mcH(\Lambda)e(\mcC)$};
\node at (4.5,0){$\mcH_q(\Lambda)e(\mcC)$};
\node at (2.9,.2){$\cong$};
\node at (2.8,.7){$\cong$};
\node at (2.8,-.8){$\cong$};

\draw[left hook-latex](-.2,0)--(-.8,0);
\draw[->>](.2,0)--(1,0);
\draw[->](2.1,0)--(3.5,0);
\draw[->](2.1,.1)--(3.5,.8);
\draw[->](2.1,-.1)--(3.5,-.8);
\end{tikzpicture}
\]
The algebra $\mathcal{L}$, called the Lusztig algebra, is some extension of a degenerate affine Hecke algebra or a non-degenerate affine Hecke algebra. This algebra contains all corresponding rational functions, which allow us to construct the KLR-basis in a unified way either in the degenerate case or in the non-degenerate case. The algebra $\tilde{\mathcal{L}}$, called the semi-rationalization algebra, can be viewed as a subalgebra of $\mathcal{L}$ only containing part of the rational functions. The semi-rationalization algebra $\tilde{\mathcal{L}}$ has three bases: the first one is the KLR basis by which we construct an isomorphism between the cyclotomic semi-rationalization algebra $\tilde{\mathcal{L}}(\Lambda)$ and the cyclotomic KLR algebra $\mcR(\Lambda)$; the second one is the degenerate Bernstein-Zelevinski basis by which we give an isomorphism between $\tilde{\mathcal{L}}(\Lambda)$ and the cyclotomic degenerate affine Hecke algebra $\mcH(\Lambda)$; the third one is the non-degenerate Bernstein-Zelevinski basis by which we give an isomorphism between $\tilde{\mathcal{L}}(\Lambda)$ and the cyclotomic non-degenerate affine Hecke algebra $\mcH_q(\Lambda)$.

\section{Preliminaries}

\subsection{The Demazure operator}
Let $S_n$ be the symmetric group with basic transpositions $\sigma_1,\cdots, \sigma_{n-1}$. Let $\Bbbk$ be a field. Then $S_n$ acts on the left on the polynomial ring $\Bbbk[X]$ and on the Laurent polynomial ring $\Bbbk[X^{\pm 1}]$ in $X:=X_1,\cdots,X_n$ by permuting variables. Using the $S_n$-action above, the {\itshape Demazure operators} are defined as $\parti_r$ on $\Bbbk[x]$ for all $1\leq r<n$ as $$\parti_r(f)=\frac{\sigma_r(f)-f}{X_{r}-X_{r+1}}.$$  
It is well-known that the Demazure operators satisfy the Leibniz rule
\begin{equation*}\label{deLeibniz}\parti_r(fg)=\parti_r(f)g+\sigma_r(f)\parti_r(g),
\end{equation*}
for all $f,g\in \Bbbk[X]$ and for $1\leq r<n$, and the relations
$$\sigma_r(\parti_r(f))=\parti_r(f), \quad \parti_r(\sigma_r(f))=-\parti_r(f).$$

Let $\Bbbk(X)$ be the corresponding rational function field, then the $S_n$-action on $\Bbbk[X]$ above can be extended to an action $w\colon \tfrac{f}{g}\mapsto \tfrac{w(f)}{w(g)}$ of $S_n$ on $\Bbbk(X)$ (by the field automorphism). This means that the action of the Demazure operators on $\Bbbk[X]$ also extends to operators on $\Bbbk(X)$.

Let $F$ be the quotient field of the subalgebra $$Z:=\Bbbk[X]^{S_n}=\{f\in \Bbbk[X] | w(f)=f \ \mbox{for every}\ w\in S_n\}$$ of $S_n$-invariants. By \cite[3.12 (a)]{Lusztig89}, there is a natural $\Bbbk$-algebra isomorphism 
\begin{equation} \label{K(X)} \Bbbk[X]\otimes_{Z} F\to \Bbbk(X), f\otimes g\mapsto fg. \end{equation} 


\subsection{Index sets} Let $I=\mathbb{Z}/e\mathbb{Z}=\{0,\cdots, e-1\}$ ($e=0\ \mbox{or} \ 2\leq e\in\mathbb{Z}$) be the vertex set of the quiver of type $A_\infty$ if $e=0$ or $A_{e-1}^{(1)}$ if $e\geq 2$:
\[
\begin{tikzpicture}
\draw[->] (-4.6,0)--(-4,0);\draw[->](-3.2,0)--(-2.6,0);\draw[->](-1.8,0)--(-1.2,0);\draw[->](-.8,0)--(-.2,0);\draw[->](.2,0)--(.8,0);
\draw[->](1.2,0)--(1.8,0);
\node at (-7,0) {$A_\infty:$};\node at (-5,0) {$\cdots$};\node at (-3.6,0) {$-2$};\node at (-2.2,0){$-1$};\node at (-1,0) {$0$};
\node at (0,0) {$1$};\node at (1,0){$2$};\node at (2.2,0){$\cdots$};

\draw[->](-5,-1.4)--(-4.5,-1.4);\draw[->](-4.5,-1.6)--(-5,-1.6);\draw[->](-3.3,-1.8)--(-2.9,-1.2);\draw[->](-2.5,-1.2)--(-2.1,-1.8);\draw[->](-2.3,-2)--
(-3.1,-2);\draw[->](-1,-1.8)--(-1,-1.2);\draw[->](-.8,-1)--(-.2,-1);\draw[->](0,-1.3)--(0,-1.8);\draw[->](-.2,-2)--(-.8,-2);\draw[->](1.2,-2.2)--(1,-1.8);
\draw[->](1.1,-1.4)--(1.5,-1.1);\draw[->](1.9,-1.1)--(2.3,-1.4);\draw[->](2.4,-1.8)--(2.2,-2.2);\draw[->](1.9,-2.4)--(1.5,-2.4);
\node at (-7,-1.5){$A_{e-1}^{(1)}:$};\node at(-5.2,-1.5){$0$};\node at(-4.3,-1.5){$1$};\node at(-3.5,-2){$2$};\node at(-2.7,-1){$0$};\node at(-1.9,-2){$1$};\node at(-1,-2){$3$};\node at(-1,-1){$0$};\node at (0,-1){$1$};\node at(0,-2){$2$};\node at(1.3,-2.4){$3$};\node at(.9,-1.6){$4$};
\node at(1.7,-1){$0$};\node at(2.5,-1.6){$1$};\node at(2.1,-2.4){$2$};\node at (3.5,-1.6){$\cdots$};
\end{tikzpicture}
\]
For $1\leq r<n$, we define the map $\sigma_r(\bi)_s=i_{\sigma_r(s)}$.
Then $S_n$ acts on the set of $n$-tuples $\bi=(i_1,\cdots, i_n)\in I^n$ by the place permutation $:w(\bi)_s=i_{w^{-1}(s)}$. 

Throughout this paper, We fix an $S_n$-orbit $\mathcal{C}$ of $I^n$.

\section{Affine Hecke algebras and their rationalizations}
\subsection{The degenerate affine Hecke algebra}
Following \cite{Drinfeld}, the {\itshape degenerate affine Hecke algebra}  $\mcH$ of $S_n$ is defined to be the associated unital $\Bbbk$-algebra with generators $\{X_1, \cdots, X_n, T_1,\cdots, T_{n-1}\}$ subject to the following relations for all admissible indices:
\begin{align}
\label{deaffhecke X}
X_rX_s&=X_sX_r;
\\
\label{deTX}
T_r X_{s} &=\sigma_{r}(X_s)T_r+\parti_r(X_s);\\
\label{deT2}
\quad T_r^2&=1;
\\
\label{deTT}
T_r  T_s &= T_s  T_r \hspace{43.4mm} \text{if $|r-s|>2$};
\\
\label{deTTT}
T_r T_{r+1}T_r &= T_{r+1}T_rT_{r+1};
\end{align}

For $w=\sigma_{r_1} \sigma_{r_2}\cdots \sigma_{r_m}\in S_n$ a reduced expression we put $T_w:=T_{r_1}T_{r_2}\cdots T_{r_m}$. Then $T_w$ is a well-defined element in $\mcH$ and the algebra $\mcH$ is a free $\Bbbk[X]$-module with the basis $\{T_w\ | \ w\in S_n\}$.

\subsection{The non-degenerate affine Hecke algebra} Assume $1\neq q\in \Bbbk^{\times}$. We define the {\itshape affine Hecke algebra} $\mcH_{q}$ to be the associated unital $\Bbbk$-algebra with generators $\{X_1^{\pm 1}, \cdots, X_{n}^{\pm1}, T_1,\cdots, T_{n-1}\}$ subject to the following relations for all admissible indices:
\begin{align}
\label{relation:affhecke X}
X_r^{\pm1}X_s^{\pm1}&=X_s^{\pm1}X_r^{\pm1},\quad X_rX_r^{-1}=X_r^{-1}X_r=1;
\\
\label{TX}
T_r X_{s} &=\sigma_r(X_s)T_r+(q-1)X_{r+1}\parti_r(X_s);\\
\label{relation:affhecke Tq}
T_r^2&=(q-1)T_r+q;
\\
\label{relation:affinehecke_T}
T_r  T_s &= T_s  T_r \hspace{43.4mm} \text{if $|r-s|>2$};
\\
\label{TTT}
T_r T_{r+1}T_r &= T_{r+1}T_rT_{r+1};
\end{align}
Similar to the degenerate case, $\mcH_q$ is a free $\Bbbk[X^{\pm1}]$-module with the basis $\{T_w\ | \ w\in S_n\}$.

\subsection{The rationalization} By \cite[Proposition 3.11]{Lusztig89}, the center of $\mcH$ is the commutative algebra $Z$ of $S_n$-invariants. Thus $\mcH$ can be seen as a $Z$-subalgebra (identified with the subspace $\mcH\otimes 1$) of the $F$-algebra $\mcH_F:=\mcH\otimes_{Z}F$. 

Similarly, the non-degenerate affine Hecke algebra $\mcH_{q}$ also can be seen as a $Z$-subalgebra (identified with the subspace $\mcH_{q}\otimes 1$) of the $F$-algebra $\mcH_{q,F}:=\mcH_{q}\otimes_{\mcZ}\mcF.$ For any $1\leq r<n$ and $f\in \Bbbk(X)$, as a consequence of \cite[3.12 (d)]{Lusztig89}, we have
\begin{align} \label{fT} T_rf=\begin{cases}
\sigma_r(f)T_r+\parti_r(f) & \text{if $f\in \mcH_F$},\\
\sigma_r(f)T_r+(q-1)X_{r+1}\parti_r(f) & \text{if $f\in \mcH_{q,F}$}.
\end{cases}\end{align}

In light of the Bernstein and Zelevinski basis of $\mcH$ and $\mcH_q$, there are decompositions, see \cite[3.12(c)]{Lusztig89}:
\begin{equation}\label{deTwdecom}\mcH_{F} (\mbox{resp.}, \mcH_{q,F})=\oplus_{w\in S_n}T_w\Bbbk(x)=\oplus_{w\in S_n}\Bbbk(x)T_w. \end{equation}

\subsection{Intertwining elements} For $1\leq r<n$, we define the {\itshape intertwining element} $\kappa_r$ as follows:
$$\kappa_r:=\begin{cases}
T_r+\tfrac{1}{X_r-X_{r+1}} & \text{in $\mcH_F$},\\
T_r+\tfrac{(q-1)X_{r+1}}{X_r-X_{r+1}} & \text{ in $\mcH_{q,F}$}.
\end{cases}$$  
The elements $\kappa_r$ have the following properties \cite[Proposition 5.2]{Lusztig89}:
\begin{align}\label{fkappa} \kappa_rf&=\sigma_r(f)\kappa_r \hspace{15mm}  \text{$\forall f\in \Bbbk(X)$};\\
\label{kappa^2}\kappa_r^2&=\begin{cases}
1-\tfrac{1}{(X_r-X_{r+1})^{2}}& \text{in $\mcH_F$},\\
q-\tfrac{(q-1)^2X_rX_{r+1}}{(X_r-X_{r+1})^2}& \text{ in $\mcH_{q,F}$};
\end{cases}\\
\label{kappa2}\kappa_r\kappa_s&=\kappa_s\kappa_r \hspace{20mm} \text{if $s\neq r,r+1$};\\
\label{kappa3}\kappa_r\kappa_{r+1}\kappa_r&=\kappa_{r+1}\kappa_r\kappa_{r+1}.
\end{align}
If $w=\sigma_{r_1} \sigma_{r_2}\cdots \sigma_{r_m}$ is a reduced expression in $S_n$, then we can define $\kappa_w:=\kappa_{r_1}\kappa_{r_2}\cdots \kappa_{r_m}$ in $\mcH_F$ (resp., $\mcH_{q,F}$). It is also a well-defined element by the braid relations (\ref{kappa2}) and (\ref{kappa3}). By the decompositions (\ref{deTwdecom}), we know that $\{\kappa_w \ | \ w\in S_n\}$ is a basis of $\mcH_F$ (resp., $\mcH_{q,F}$) as free $\Bbbk(X)$-module.

\section{The Lusztig extensions of affine Hecke algebras}
From now on, we fix $I=\mathbb{Z}/e\mathbb{Z}$, where $e=\mathrm{char}\Bbbk$ for the degenerate affine Hecke algebras, and $e$ is the smallest positive integer such that $1+q+\cdots +q^{e-1}=0$ and setting $e=0$ if no such integer exists for the non-degenerate affine Hecke algebras. 

\subsection{The Lusztig extensions}  Let $\mathcal{E}$ be the unital $\Bbbk$-algebra with basis $\{\epsilon(\bi) | \bi\in \mcC\}$. Multiplication is given by 
\begin{equation} \label{epsilon} \epsilon(\bi)\epsilon(\bj)=\delta^\bi_\bj\epsilon(\bi).\end{equation}
The {\it Lusztig extension} of $\mcH$ with respect to $\mathcal{E}$ is the $\Bbbk$-algebra $\mathcal{L}$ which is equal as $\Bbbk$-space to the tensor product $$\mathcal{L}:=\mcH_F\otimes_\Bbbk\mathcal{E}=\oplus_{w\in S_n,\bi\in \mcC}\kappa_w\Bbbk(X)\epsilon(\bi)$$
of the rationalization algebra $\mcH_F$ and the semi-simple algebra $\mathcal{E}$. Multiplication is defined so that $\mcH_{F}$ (identified with the subspace $\mcH_{F}\otimes 1$) and
$\mathcal{E}$ (identified with the subspace $1\otimes \mathcal{E}$) are subalgebras of $\mathcal{L}$, and in addition
   \begin{align}
 \label{Xepis}X_r\epsilon(\bi)&=\epsilon(\bi)X_r,\\
 \label{kappaepsi}\kappa_r\epsilon(\bi)&=\epsilon(\sigma_r(\bi))\kappa_r,
 \end{align}
for all $1\leq r<n$ and $\bi\in \mcC$.

The {\it Lusztig extension $\mathcal{L}_q$} of $\mcH_q$ with respect to $\mathcal{E}$ is defined similarly. We have the following relations \begin{equation}\label{deTepsi}T_r\epsilon(\bi)=\begin{cases}\epsilon(\sigma_r(\bi))T_r+\tfrac{1}{X_r-X_{r+1}}(\epsilon(\sigma_r(\bi))-\epsilon(\bi)) & \text{in $\mathcal{L}$},\\
 \epsilon(\sigma_r(\bi))T_r+\tfrac{(q-1)X_{r+1}}{X_r-X_{r+1}}(\epsilon(\sigma_r(\bi))-\epsilon(\bi)) & \text{in $\mathcal{L}_{q}$}.
 \end{cases}
 \end{equation}

\subsection{The Brundan-Kleshchev auxiliary elements}
 Recall that \cite[(3.12), (4.21)]{Brundan-Kleshchev09}, for each $1\leq r\leq n$, Brundan and Kleshchev introduced the element 
\begin{equation} \label{yX} y_r:=\begin{cases}\sum_{\bi\in \mcC}(X_r-i_r)\epsilon(\bi)&\text{in $\mathcal{L}$},\\
	\sum_{\bi\in \mcC}(1-q^{-i_r}X_r)\epsilon(\bi)&\text{in $\mathcal{L}_{q}$}.
	\end{cases} 
	\end{equation} Then $y_r\epsilon(\bi)=\epsilon(\bi)y_r$ by \ref{Xepis}, and $y_r$ is a unit in $\mathcal{L}$ with $$y_r^{-1}=\begin{cases}\sum_{\bi\in \mcC}(X_r-i_r)^{-1}\epsilon(\bi)&\text{in $\mathcal{L}$},\\
	\sum_{\bi\in \mcC}(1-q^{-i_r}X_r)^{-1}\epsilon(\bi)&\text{in $\mathcal{L}_{q}$}.
	\end{cases} $$
Let $\Bbbk[y]$ be the polynomial ring with $y:=y_1,y_2,\cdots,y_n$ and $\Bbbk(y)$ be the rational function field. 
\begin{lemma}\label{lem:K(X)=K(y)} $\Bbbk(X)\otimes_\Bbbk\mathcal{E}=\Bbbk(y)\otimes_\Bbbk\mathcal{E}$ in $\mathcal{L}$ or $\mathcal{L}_q$ as subalgebras. 
\end{lemma}
\begin{proof} For the degenerate case, if $g(y)$ is a polynomial in $\Bbbk[y]$, then $$g(y)=\sum_{\bi \in \mcC}g(y)\epsilon(\bi)=\sum_{\bi \in \mcC}g(X_1-i_1,\cdots,X_n-i_n)\epsilon(\bi)$$
	in $\mathcal{L}(\mcH,F,\mcC)$. Therefore, if $g(y)\neq 0$, then $g(X_1-i_1,\cdots,X_n-i_n)\neq 0$ in $\Bbbk[X]$ for all $\bi\in \mcC$. Thus $g(y)^{-1}=\sum_{\bi \in \mcC}g(X_1-i_1,\cdots,X_n-i_n)^{-1}\epsilon(\bi)$ exists in $\mathcal{L}$. So $\Bbbk(y)\otimes_\Bbbk\mathcal{E}$ is a subalgebra of $\mathcal{L}$.
	
	Combining (\ref{yX}) with $ X_r=\sum_{\bi\in \mcC}(y_r+i_r)\epsilon(\bi)$,	we know that $\Bbbk(X)\otimes_\Bbbk\mathcal{E}=\Bbbk(y)\otimes_\Bbbk\mathcal{E}$ in $\mathcal{L}$.
	
Imitate the proof above, we can prove the non-degenerate case.
\end{proof}

For each $1\leq r<n$, Brundan and Kleshchev defined the element
\begin{equation*}\label{deqr(i)} Q_r(\bi)=\begin{cases} 1+y_{r+1}-y_r & \text{if $i_r=i_{r+1}$},\\
1& \text{if $i_r-i_{r+1}=1, e\neq 2$},\\
\tfrac{2+y_{r+1}-y_r}{(1+y_{r+1}-y_r)^{2}}& \text{if $i_r-i_{r+1}=-1, e\neq 2$},\\
\tfrac{1}{1+y_{r+1}-y_{r}}& \text{if $i_r-i_{r+1}=1, e=2$},\\
1-\tfrac{1}{y_r-y_{r+1}+i_r-i_{r+1}}& \text{if $i_r-i_{r+1}\neq 0, \pm 1$}
\end{cases}
\end{equation*}
in $\mathcal{L}$,
and
\begin{equation*}\label{Qr(i)} Q_r(\bi)=\begin{cases} 1-q+qy_{r+1}-y_r & \text{if $i_r=i_{r+1}$},\\
q^{i_r} & \text{if $i_r-i_{r+1}=1, e\neq 2$},\\
\frac{1-q^2+q^2y_{r+1}-y_r}{q^{i_{r}}(1-q+qy_{r+1}-y_r)^{2}}& \text{if $i_r-i_{r+1}=-1, e\neq 2$},\\
\frac{1}{1-q+qy_{r+1}-y_r} & \text{if $i_r-i_{r+1}=1, e=2$},\\
\frac{q^{i_r-i_{r+1}}(1-y_r)-q(1-y_{r+1})}{q^{i_r-i_{r+1}}(1-y_r)-(1-y_{r+1})}& \text{if $i_r-i_{r+1}\neq 0, \pm 1$},
\end{cases}
\end{equation*}
in $\mathcal{L}_q$.

Notice that the symmetric group $S_n$ acts on the left on $\Bbbk[y]$ by permuting variables:
	$$\sigma_r(y_s)=\begin{cases}\sum_{\bi\in \mcC}(\sigma_r(X_s)-\bi_s)\epsilon(\sigma_r(\bi))&\text{in $\mathcal{L}$}\\
	\sum_{\bi\in \mcC}(1-q^{-i_r}\sigma_r(X_s))\epsilon(\sigma_r(\bi))&\text{in $\mathcal{L}_{q}$}
	\end{cases} =y_{\sigma_r(s)},$$
 where $1\leq r<n$ and $1\leq s\leq n$. 

The following Lemma is proved in \cite[(3.27)-(3.29), (4.33)-(4.35)]{Brundan-Kleshchev09}.
\begin{lemma} \label{lem:QsrQsr}Let $1\leq r<n$ and $\bi\in \mcC$. Then
	\begin{equation*} \sigma_r(Q_r(\sigma_r(\bi)))Q_r(\bi)=\begin{cases}1-(y_{r+1}-y_{r})^2 & \text{if $i_r-i_{r+1}=0$};\\
	\frac{2+y_r-y_{r+1}}{(1+y_r-y_{r+1})^{2}}	&\text{if $i_r-i_{r+1}=1, e\neq 2$};\\
	\frac{2+y_{r+1}-y_r}{(1+y_{r+1}-y_r)^{2}}	&\text{if $i_r-i_{r+1}=-1, e\neq2$};\\
	\frac{1}{(1+y_{r+1}-y_r)^{2}}	&\text{if $i_r-i_{r+1}=1, e=2$};\\
	1-\frac{1}{(y_r-y_{r+1}+i_r-i_{r+1})^{2}}& \text{if $i_r-i_{r+1}\neq 0, \pm 1$}.
	\end{cases}
	\end{equation*}
	in $\mathcal{L}$, and in $\mathcal{L}_q$,
	\begin{align*} &\sigma_r(Q_r(\sigma_r(\bi)))Q_r(\bi)\notag\\
	&=\begin{cases}(1-q+qy_{r+1}-y_r)(1-q+qy_r-y_{r+1}) & \text{if $i_r=i_{r+1}$};\\
	\tfrac{q(1-q^2+q^2y_r-y_{r+1})}{(1-q+qy_r-y_{r+1})^{2}}&\text{if $i_r-i_{r+1}=1, e\neq 2$};\\
	\tfrac{q(1-q^2+q^2y_{r+1}-y_{r})}{(1-q+qy_{r+1}-y_{r})^{2}}&\text{if $i_r-i_{r+1}=-1, e\neq2$};\\
	\tfrac{1}{(1-q+qy_r-y_{r+1})(1-q+qy_{r+1}-y_r)}&\text{if $i_r-i_{r+1}=1, e=2$};\\
	q-\tfrac{(1-q)^2q^{i_r+i_{r+1}}(1-y_{r+1})(1-y_r)}{[q^{i_{r+1}}(1-y_{r+1})-q^{i_r}(1-y_r)]^2}& \text{if $i_r\neq 0, \pm 1$}.
	\end{cases}
	\end{align*}

\end{lemma}

Surprisingly, either in $\mathcal{L}$ or in $\mathcal{L}_q$, we can set the same named elements $$\theta_r:=\kappa_r\sum_{\bi\in \mcC}Q_r^{-1}(\bi)\epsilon(\bi),$$
 which share the same relations. 
\begin{lemma} For each $1\leq r<n$ and $\bi\in \mcC$, we have
		\begin{align} \label{thetaepsilon} \theta_r\epsilon(\bi)&=\epsilon(\sigma_r(\bi))\theta_r,\\
	\label{ftheta} f\theta_r&=\theta_r\sigma_r(f) \hspace{11mm}\forall f\in \Bbbk(y),\\
	\label{theta2}\theta_r\theta_s&=\theta_s\theta_r \hspace{19mm} \text{if $s\neq r,r+1$},\\
	\label{theta3}\theta_r\theta_{r+1}\theta_r&=\theta_{r+1}\theta_r\theta_{r+1},
	\end{align}
	\begin{equation}\label{theta^2}\theta_r^2\epsilon(\bi)=\begin{cases}-\frac{1}{(y_{r+1}-y_r)^2}\epsilon(\bi) & \text{if $i_r=i_{r+1}$},\\
	(y_r-y_{r+1})\epsilon(\bi) &\text{if $i_r-i_{r+1}=1, e\neq 2$},\\
	(y_{r+1}-y_r)\epsilon(\bi) & \text{if $i_r-i_{r+1}=-1, e\neq 2$},\\
	(y_r-y_{r+1})(y_{r+1}-y_r)\epsilon(\bi)& \text{if $i_r-i_{r+1}=1, e=2$},\\
	\epsilon(\bi)& \text{if $i_r-i_{r+1}\neq0,\pm1$}.
	\end{cases}
	\end{equation}
\end{lemma}
\begin{proof}
	The assertions follow straightforwardly from (\ref{fkappa})-(\ref{kappa3}), (\ref{thetaepsilon}), Lemma \ref{lem:QsrQsr} and (\ref{kappa^2}). 
\end{proof}

\subsection{The KLR basis of the Lusztig extensions} For each $1\leq r<n$, set \begin{equation} \label{psitheta}\psi_r=\sum_{\bi\in \mcC}[\theta_r-\tfrac{\delta_{i_r}^{i_{r+1}}}{y_r-y_{r+1}}]\epsilon(\bi).
\end{equation}
\begin{proposition} \label{prop:KLR basis} The Lusztig extension $\mathcal{L}$ or $\mathcal{L}_q$ is generated by $$\{y_1,\cdots,y_n,\psi_1,\cdots,\psi_{n-1},f^{-1}, \epsilon(\bi)\ | \ 0\neq f\in \Bbbk[y], \bi\in \mcC\}$$ subject to the following relations.
	\begin{align}\label{ff^{-1}} ff^{-1}&=f^{-1}f=1, \forall \quad 0\neq f\in \Bbbk[y];\\
		\label{yepsilon}y_r\epsilon(\bi)&=\epsilon(\bi)y_r, \quad y_ry_s=y_sy_r;\\	
	\label{psiepsilon(i)} \psi_r\epsilon(\bi)&=\epsilon(\sigma_r(\bi))\psi_r;\\
	\label{psiyepsilon(i)}\psi_ry_s\epsilon(\bi)&=[\sigma_r(y_s)\psi_r+\delta_{i_{r+1}}^{i_r}\parti_r(y_s)]\epsilon(\bi);\\
	\label{psi^2}\psi_r^2\epsilon(\bi)&=\begin{cases}0 & \text{if $i_r-i_{r+1}=0$},\\
	(y_r-y_{r+1})\epsilon(\bi) &\text{if $i_r-i_{r+1}=1, e\neq 2$},\\
	(y_{r+1}-y_r)\epsilon(\bi) & \text{if $i_r-i_{r+1}=-1, e\neq 2$},\\
	-(y_r-y_{r+1})^2\epsilon(\bi)& \text{if $i_r-i_{r+1}=1, e=2$},\\
	\epsilon(\bi)& \text{if $i_r\neq0,\pm1$}.
	\end{cases}\\
	\psi_r\psi_s&=\psi_s\psi_r,\hspace{26mm} \text{if $s\neq r,r+1$};
	\end{align}
	\begin{align}
	\label{psi3}&(\psi_r\psi_{r+1}\psi_r-\psi_{r+1}\psi_r\psi_{r+1})\epsilon(\bi)\notag\\
	&=\begin{cases}-\epsilon(\bi) & \text{if $i_{r+2}=i_r, i_r-i_{r+1}=1, e\neq 2$},\\
	\epsilon(\bi) &\text{if $i_{r+2}=i_r, i_{r}-i_{r+1}=-1, e\neq 2$},\\
	-(y_r+y_{r+2}-2y_{r+1})\epsilon(\bi) & \text{if $i_{r+2}=i_r, i_{r}-i_{r+1}=1, e=2$},\\
	0 & \text{else}.
	\end{cases}
	\end{align}	
\end{proposition}
\begin{proof} The relations (\ref{ff^{-1}}) and (\ref{yepsilon}) can be proved by straightforward calculations using Lemma \ref{lem:K(X)=K(y)}.
	
	As a result of (\ref{thetaepsilon}) and (\ref{psitheta}), we see that $$\psi_r\epsilon(\bi)=(\theta_r-\tfrac{\delta_{i_{r+1}}^{i_r}}{y_r-y_{r+1}})\epsilon(\bi)=\epsilon(\sigma_r(\bi))(\theta_r-\tfrac{\delta_{i_{r+1}}^{i_r}}{y_r-y_{r+1}})=\epsilon(\sigma_r(\bi))\psi_r,$$
	where the third identity holds since $\sigma_r(\bi)=\bi$ whenever $i_r=i_{r+1}$. 
	
	By (\ref{psitheta}) and (\ref{ftheta}), we obtain
	\begin{align*}\psi_ry_s\epsilon(\bi)&=(\theta_r-\tfrac{\delta_{i_{r+1}}^{i_r}}{y_r-y_{r+1}})y_s\epsilon(\bi)\\
	&=(\sigma_r(y_s)\theta_r-\tfrac{\delta_{i_{r+1}}^{i_r}y_s}{y_r-y_{r+1}})\epsilon(\bi) \\
	&=(\sigma_r(y_s)\psi_r+\tfrac{\delta_{i_{r+1}}^{i_r}\sigma_r(y_s)}{y_r-y_{r+1}}-\tfrac{\delta_{i_{r+1}}^{i_r}y_s}{y_r-y_{r+1}})\epsilon(\bi)\\
	&=(\sigma_r(y_s)\psi_r+\delta_{i_{r+1}}^{i_r}\parti_r(y_s))\epsilon(\bi).
	\end{align*}

	If $s\neq r,r+1$, using (\ref{psiepsilon(i)}), (\ref{psitheta}), (\ref{ftheta}) and (\ref{theta2}), then we have 
	\begin{align*}\psi_r\psi_s\epsilon(\bi)&=\psi_r\epsilon(\sigma_s(\bi))\psi_s\epsilon(\bi)\\
	&=(\theta_r-\tfrac{\delta_{i_{r+1}}^{i_r}}{y_r-y_{r+1}})(\theta_s-\tfrac{\delta_{i_{s+1}}^{i_s}}{y_s-y_{s+1}})e(\bi)\\
	&=(\theta_r\theta_s-\tfrac{\delta_{i_{r+1}}^{i_r}}{y_r-y_{r+1}}\theta_s-\theta_r\tfrac{\delta_{i_{s+1}}^{i_s}}{y_s-y_{s+1}}+\tfrac{\delta_{i_{r+1}}^{i_r}\delta_{i_{s+1}}^{i_s}}{(y_r-y_{r+1})(y_s-y_{s+1})})\epsilon(\bi)\\
	&=(\theta_s\theta_r-\theta_s\tfrac{\delta_{i_{r+1}}^{i_r}}{y_r-y_{r+1}}-\tfrac{\delta_{i_{s+1}}^{i_s}}{y_s-y_{s+1}}\theta_r+\tfrac{\delta_{i_{r+1}}^{i_r}\delta_{i_{s+1}}^{i_s}}{(y_r-y_{r+1})(y_s-y_{s+1})})\epsilon(\bi)\\
	&=(\theta_s-\tfrac{\delta_{i_{s+1}}^{i_s}}{y_s-y_{s+1}})(\theta_r-\tfrac{\delta_{i_{r+1}}^{i_r}}{y_r-y_{r+1}})\epsilon(\bi)	\\
	&=\psi_s\psi_r\epsilon(\bi)
	\end{align*}
	So $\psi_r\psi_s=\psi_s\psi_r$.	
	
	As a result of (\ref{psitheta}), (\ref{psiepsilon(i)}) and Lemma \ref{theta^2}, we get
		\begin{align*}
	\psi_r^2\epsilon(\bi)&=\psi_r\epsilon(\sigma_r(\bi))\psi_r\epsilon(\bi)\\
	&=(\theta_r-\tfrac{\delta_{i_{r}}^{i_{r+1}}}{y_r-y_{r+1}})(\theta_r-\tfrac{\delta_{i_{r+1}}^{i_r}}{y_r-y_{r+1}})\epsilon(\bi)\\
	&=(\theta_r^2+\tfrac{\delta_{i_{r+1}}^{i_r}}{(y_r-y_{r+1})^2})\epsilon(\bi)\\
	&=\begin{cases}0 & \text{if $i_r-i_{r+1}=0$},\\
	(y_r-y_{r+1})\epsilon(\bi) &\text{if $i_r-i_{r+1}=1, e\neq 2$},\\
	(y_{r+1}-y_r)\epsilon(\bi) & \text{if $i_r-i_{r+1}=-1, e\neq 2$},\\
	(y_r-y_{r+1})(y_{r+1}-y_r)\epsilon(\bi)& \text{if $i_r-i_{r+1}=1, e=2$},\\
	\epsilon(\bi)& \text{if $i_r\neq0,\pm1$}.
	\end{cases}
	\end{align*}

	In the light of (\ref{psiepsilon(i)}), (\ref{psitheta}) and (\ref{ftheta}), there holds that
	\begin{align*}&\psi_r\psi_{r+1}\psi_r\epsilon(\bi)\\
	&=(\theta_r-\tfrac{\delta_{i_{r+2}}^{i_{r+1}}}{y_r-y_{r+1}})(\theta_{r+1}-\tfrac{\delta_{i_{r+2}}^{i_{r}}}{y_{r+1}-y_{r+2}})(\theta_r-\tfrac{\delta_{i_{r+1}}^{i_{r}}}{y_r-y_{r+1}})\epsilon(\bi)\\
	&=[\theta_r\theta_{r+1}\theta_r-\tfrac{\delta_{i_{r+1}}^{i_{r}}}{y_{r+1}-y_{r+2}}\theta_r\theta_{r+1}-\tfrac{\delta_{i_{r+2}}^{i_{r+1}}}{y_r-y_{r+1}}\theta_{r+1}\theta_r+\tfrac{\delta_{i_{r+1}}^{i_{r}}\delta_{i_{r+2}}^{i_{r+1}}}{(y_r-y_{r+2})(y_{r+1}-y_{r+2})}\theta_r\\
	&+\tfrac{\delta_{i_{r+1}}^{i_{r}}\delta_{i_{r+2}}^{i_{r+1}}}{(y_r-y_{r+1})(y_{r}-y_{r+2})}\theta_{r+1}-\tfrac{\delta_{i_{r+2}}^{i_{r}}}{y_{r}-y_{r+2}}\theta_r^2-\tfrac{\delta_{i_{r+1}}^{i_{r}}\delta_{i_{r+2}}^{i_{r}}\delta_{i_{r+2}}^{i_{r+1}}}{(y_r-y_{r+1})^2(y_{r+1}-y_{r+2})}]\epsilon(\bi)
	\end{align*}
	Similarly, we have
	\begin{align*}&\psi_s\psi_r\psi_s\epsilon(\bi)\\
	&=(\theta_{r+1}-\tfrac{\delta_{i_{r+1}}^{i_{r}}}{y_{r+1}-y_{r+2}})(\theta_r-\tfrac{\delta_{i_{r+2}}^{i_{r}}}{y_r-y_{r+1}})(\theta_{r+1}-\tfrac{\delta_{i_{r+2}}^{i_{r+1}}}{y_{r+1}-y_{r+2}})\epsilon(\bi)\\
	&=[\theta_{r+1}\theta_{r}\theta_{r+1}-\tfrac{\delta_{i_{r+1}}^{i_{r}}}{y_{r+1}-y_{r+2}}\theta_r\theta_{r+1}-\tfrac{\delta_{i_{r+2}}^{i_{r+1}}}{y_r-y_{r+1}}\theta_{r+1}\theta_r+\tfrac{\delta_{i_{r+1}}^{i_{r}}\delta_{i_{r+2}}^{i_{r+1}}}{(y_r-y_{r+2})(y_{r+1}-y_{r+2})}\theta_r\\
	&+\tfrac{\delta_{i_{r+1}}^{i_{r}}\delta_{i_{r+2}}^{i_{r+1}}}{(y_r-y_{r+1})(y_{r}-y_{r+2})}\theta_{r+1}-\tfrac{\delta_{i_{r+2}}^{i_{r}}}{y_{r}-y_{r+2}}\theta_{r+1}^2-\tfrac{\delta_{i_{r+1}}^{i_{r}}\delta_{i_{r+2}}^{i_{r}}\delta_{i_{r+2}}^{i_{r+1}}}{(y_{r}-y_{r+1})(y_{r+1}-y_{r+2})^2}]\epsilon(\bi)
	\end{align*}
	By (\ref{theta3}) and \ref{theta^2}, we arrive at 
	\begin{align*}
	&[\psi_r\psi_{r+1}\psi_r-\psi_{r+1}\psi_r\psi_{r+1}]\epsilon(\bi)\\
	&=\delta_{i_{r+2}}^{i_r}[\tfrac{1}{y_r-y_{r+2}}(\theta_{r+1}^2-\theta_r^2)+\tfrac{\delta_{i_{r+1}}^{i_r}(y_r+y_{r+2}-2y_{r+1})}{(y_{r}-y_{r+1})^2(y_{r+1}-y_{r+2})^2}]\epsilon(\bi)\\
	&=\begin{cases}-\epsilon(\bi) & \text{if $i_{r+2}=i_r, i_r-i_{r+1}=1, e\neq 2$},\\
	\epsilon(\bi) &\text{if $i_{r+2}=i_r, i_{r}-i_{r+1}=-1, e\neq 2$},\\
	-(y_r+y_{r+2}-2y_{r+1})\epsilon(\bi) & \text{if $i_{r+2}=i_r, i_{r}-i_{r+1}=1, e=2$},\\
	0 & \text{else}.
	\end{cases}
	\end{align*}

	To finish the proof of the theorem, we need to prove the relations (\ref{ff^{-1}})-(\ref{psi3}) generate all relations. In fact, for each $w\in S_n$, we {\it fix} a reduced decomposition $w=\sigma_{r_1} \sigma_{r_2}\cdots \sigma_{r_m}$ and define the element
	$$\psi_w:=\psi_{r_1}\psi_{r_2}\cdots \psi_{r_m}\in \mathcal{L}.$$
	Note that $\psi_w$ in general does depend on the choice of reduced decomposition of $w$ \cite[Proposition 2.5]{BKW}. Since $$\psi_r=\theta_r-\sum_{\bi\in \mcC}\tfrac{\delta_{i_r}^{i_{r+1}}}{y_r-y_{r+1}}\epsilon(\bi)=\kappa_r\sum_{\bi\in \mcC}Q_r^{-1}(\bi)\epsilon(\bi)-\sum_{\bi\in \mcC}\tfrac{\delta_{i_r}^{i_{r+1}}}{y_r-y_{r+1}}\epsilon(\bi)$$ and $\{\kappa_w\ | \ w\in S_n\}$ is a $\Bbbk(x)$-basis of $\mcH_F$, we can show that $\{\psi_w\ | \ w\in S_n\}$ is a $\Bbbk(y)\otimes_\Bbbk\mathcal{E}$-basis of $\mathcal{L}$ by Lemma \ref{lem:K(X)=K(y)} since every element in $\mathcal{L}$ can be written as $\sum_{w\in S_n,\bi\in \mcC}\psi_wf_{w,\bi}(y)\epsilon(\bi)$ with $f_{w,\bi}(y)$ in $\Bbbk(y)$ by the relations (\ref{ff^{-1}})-(\ref{psi3}). Therefore
	the generating set
	$$\{y_1,\cdots,y_n,\psi_1,\cdots,\psi_{n-1},f^{-1}, \epsilon(\bi)\ | \ 0\neq f\in \Bbbk[y], \bi\in \mcC\}$$
	subject to relations (\ref{ff^{-1}})-(\ref{psi3}) is complete. 	
	
	The same proof above is valid for $\mathcal{L}_q$, and then we are done.	
\end{proof}
According to the Proposition above, we will not distinguish between Lusztig extensions in the sequel.

\subsection{The cyclotomic KLR algebras} Recall \cite[Subsection 2.2]{Brundan-Kleshchev09}, the {\it KLR algebra of type A} is defined to be the $\Bbbk$-algebra $\mcR$ generated by $$\{y_1,\cdots,y_n, \psi_1,\cdots,\psi_n, \epsilon(\bi) \ |\ \bi\in \mcC,\}$$
subject to the relations (\ref{ff^{-1}})-(\ref{psi3}) of Proposition \ref{prop:KLR basis}. Thus $\mcR$ can be viewed as a subalgebra of $\mathcal{L}$. 

From now on, we fix an index $\Lambda=(\Lambda_i)_{i\in I}\in \mathbb{N}^I$ (we follow the convention that $\mathbb{N}=\{0,1,2,\cdots\}$) with $\sum_{i\in I}\Lambda_i<\infty$. We call the quotient
$$\mcR(\Lambda):=\mcR/\langle y_1^{\Lambda_{i_1}}\epsilon(\bi) | \bi\in \mcC \rangle$$
a {\it cyclotomic KLR algebra of type A}.

Recall that, Brund and Kleshchev proved the following result in \cite[Lemma 2.1]{Brundan-Kleshchev09}.
\begin{lemma} \label{lem:nilp of y} The elements $y_r$ are nilpotent in $\mcR(\Lambda)$.
\end{lemma}

\subsection{The semi-rationalizations}  
 We define the {\it semi-rationalization algebra} as the $\Bbbk$-algebra $\tilde{\mathcal{L}}$ generated by $$\{y_1,\cdots,y_n, \psi_1,\cdots,\psi_{n-1}, f^{-1}(y),\epsilon(\bi) \ |\ \bi\in \mcC, f(y)\in \Bbbk[y] \ \mbox{with} \ f(0)\neq 0\}$$
subject to the relations (\ref{ff^{-1}})-(\ref{psi3}) of Proposition \ref{prop:KLR basis}. 

The following result shows that we can construct a complete generating set from the degenerate affine Hecke algebras or from the non-degenerate affine Hecke algebras. This is the key point for proving the BK isomorphism in the sequel.

\begin{theorem}\label{thm:basis semi-ration algebra} $(1)$ Denote by $1\epsilon(\bi)=\epsilon(\bi)$. Then the algebra $\tilde{\mathcal{L}}$ is generated by $$\{X_1,\cdots,X_n, T_1,\cdots,T_{n-1}, f^{-1}(X)\epsilon(\bi) \ |\ \bi\in \mcC, f(X)\in \Bbbk[X] \ \mbox{with} \ f(\bi)\neq 0\}$$
subject to relations (\ref{deaffhecke X})-(\ref{deTTT}), (\ref{epsilon}), (\ref{Xepis}), (\ref{deTepsi}) and for $f\in \Bbbk[X]$ with $f(\bi)\neq 0$
\begin{equation}\label{def(X)f^{-1}(X)}
\epsilon(\bj)\cdot f^{-1}\epsilon(\bi)=\delta_{\bi}^{\bj}\epsilon(\bi)=f^{-1}\epsilon(\bi)\cdot \epsilon(\bj), \ f\cdot f^{-1}\epsilon(\bi)=\epsilon(\bi)=f^{-1}\epsilon(\bi)\cdot f.
\end{equation}

$(2)$ Denote by $1\epsilon(\bi)=\epsilon(\bi)$. Then the algebra $\tilde{\mathcal{L}}$ is generated by $$\{X_1,\cdots,X_n, T_1,\cdots,T_{n-1}, f^{-1}(X)\epsilon(\bi) \ |\ \bi\in \mcC, f(X)\in \Bbbk[X] \ \mbox{with} \ f(q^\bi)\neq 0\}$$
subject to relations (\ref{relation:affhecke X})-(\ref{TTT}), (\ref{epsilon}), (\ref{Xepis}), (\ref{Tepsi}) and for $f\in \Bbbk[X]$ with $f(q^\bi)\neq 0$
\begin{equation}\label{f(X)f^{-1}(X)}
\epsilon(\bj)\cdot f^{-1}\epsilon(\bi)=\delta_{\bi}^{\bj}\epsilon(\bi)=(f^{-1}\epsilon(\bi)\cdot \epsilon(\bj), \ f\cdot f^{-1}\epsilon(\bi)=\epsilon(\bi)=f^{-1}\epsilon(\bi)\cdot f.
\end{equation}
\end{theorem}
\begin{proof}
	(1) There is an obvious homomorphism $\alpha\colon \tilde{\mathcal{L}}\to \mathcal{L}$ by sending generators to the same named generators. This homomorphism is injective since using relations (\ref{ff^{-1}})-(\ref{psi3}), every element in $\tilde{\mathcal{L}}$ can be written as  
	$$\sum_{w\in S_n,\bi\in \mcC}\psi_w\tfrac{f_{w,\bi}(y)}{g_{w,\bi}(y)}\epsilon(\bi)$$ with $f_{w,\bi}(y), g_{w,\bi}(y)$ in $\Bbbk[y]$ and $g_{w,\bi}(0)\neq 0$, and $\{\psi_w\ | \ w\in S_n\}$ is a $\Bbbk(y)\otimes_\Bbbk\mathcal{E}$-basis of $\mathcal{L}$. Thus $\mathrm{Im}\alpha=\oplus_{w\in S_n}\psi_w\mathcal{P}(y,\mathcal{E})$, where $\mathcal{P}(y,\mathcal{E})$ is the commutative algebra $\{fg^{-1}\ | \ f,g\in\Bbbk[y], g(0)\neq0\}\otimes_\Bbbk\mathcal{E}$.
	
	Assume $\mathcal{A}$ is the $\Bbbk$-algebra generated by $$\{X_1,\cdots,X_n, T_1,\cdots,T_{n-1}, f^{-1}(X)\epsilon(\bi) \ |\ \bi\in \mcC, f(X)\in \Bbbk[X] \ \mbox{with} \ f(\bi)\neq 0\}$$
 subject to relations (\ref{deaffhecke X})-(\ref{deTTT}), (\ref{epsilon}), (\ref{Xepis}), (\ref{deTepsi}) and (\ref{def(X)f^{-1}(X)}). Then, using these relations, every element in $\mathcal{A}$ can be written as $$\sum_{w\in S_n,\bi\in \mcC}T_wf_{w,\bi(X)}\cdot g_{w,\bi}^{-1}(X)\epsilon(\bi)$$ with $f_{w,\bi}(X), g_{w,\bi}(X)$ in $\Bbbk[X]$ and $g_{w,\bi}(\bi)\neq 0$. Thus there also has an injective homomorphism $\alpha'\colon \mathcal{A}\to \mathcal{L}$ by sending generators to the same named generators. Therefore $\mathrm{Im}\alpha'=\oplus_{w\in S_n}T_w\mathcal{P}(X,\mathcal{E})$, where $\mathcal{P}(X,\mathcal{E})$ is the commutative algebra $\{f\cdot g^{-1}\epsilon(\bi) \ | \  \bi\in \mcC,f,g\in \Bbbk[X], g(\bi)\neq 0\}$.
 
 We claim that $\mathrm{Im}\alpha=\mathrm{Im}\alpha'$. In fact, by Lemma \ref{lem:K(X)=K(y)}, $\mathcal{P}(X,\mathcal{E})=\mathcal{P}(y,\mathcal{E})$ in $\mathcal{L}$. Notice that
 \begin{equation*}\psi_r=\sum\limits_{\begin{subarray}{1}\bi\in\mcC \\
 	i_r\neq i_{r+1}\end{subarray}}(T_r+\tfrac{1}{y_r-y_{r+1}+i_r-i_{r+1}})Q_r^{-1}(\bi)\epsilon(\bi)+\sum\limits_{\begin{subarray}{1}\bi\in\mcC \\
 	i_r=i_{r+1}\end{subarray}}(T_r+1)Q^{-1}_r(\bi)\epsilon(\bi)\end{equation*}
and $Q_r(\bi), Q_r^{-1}(\bi)\in \mathcal{P}(y,\mathcal{E})$ in $\mathcal{L}$, we get $\mathrm{Im}\alpha=\mathrm{Im}\alpha'$, and then $\mathcal{A}=\tilde{\mathcal{L}}$.
 
 (2) Imitate the proof of the statement $(1)$.
\end{proof}

\begin{lemma} \label{lem:ccover} $\tilde{\mathcal{L}}/\langle y_1^{\Lambda_{i_1}}\epsilon(\bi)\ | \ \bi\in \mcC\rangle =\begin{cases}\tilde{\mathcal{L}}/\langle \prod_{i\in I}(X_1-i)^{\Lambda_i}\rangle &\text{for $\mcH$},\\
	\tilde{\mathcal{L}}/\langle \prod_{i\in I}(X_1-q^i)^{\Lambda_i}\rangle	&\text{for $\mcH_q$}.
	\end{cases}$
\end{lemma}
\begin{proof} For the degenerate case, by the relation of $y_1$ and $X_1$, we have \begin{align*}\prod_{i\in I}(X_1-i)^{\Lambda_i}&=\sum_{\bj\in \mcC}\prod_{i\in I}(y_1+j_1-i)^{\Lambda_i}\epsilon(\bj)\\
	&=\sum_{\bj\in \mcC}\prod_{i\in I,i\neq j_1}(y_1+j_1-i)^{\Lambda_i}y_1^{\Lambda_{j_1}}\epsilon(\bj)
	\end{align*} is in $\langle y_1^{\Lambda_{j_1}}\epsilon(\bj)\ | \ \bj\in \mcC\rangle$. By Theorem (\ref{thm:basis semi-ration algebra}) (1), $\prod_{i\in I, i\neq j_1}[(X_1-i)^{\Lambda_i}]^{-1}\epsilon(\bj)$ is in $\tilde{\mathcal{L}}$, thus \begin{align*}y_1^{\Lambda_{j_1}}\epsilon(\bj)&=(X_1-j_1)^{\Lambda_{j_1}}\epsilon(\bj)\\
	&=\prod_{i\in I}(X_1-i)^{\Lambda_i}\prod_{i\in I, i\neq j_1}[(X_1-i)^{\Lambda_i}]^{-1}\epsilon(\bj)
	\end{align*}
	is in $\langle \prod_{i\in I}(X_1-i)^{\Lambda_i}\rangle$. Hence $\langle y_1^{\Lambda_{i_1}}\epsilon(\bi)\ | \ \bi\in \mcC\rangle =\langle \prod_{i\in I}(X_1-i)^{\Lambda_i}\rangle$ in $\tilde{\mathcal{L}}.$
	
	For the non-degenerate case, by the relation of $y_1$ and $X_1$, we get that \begin{align*}\prod_{i\in I}(X_1-q^i)^{\Lambda_i}&=\sum_{\bj\in \mcC}\prod_{i\in I}(q^{j_1}-q^{i}-q^{j_1}y_1)^{\Lambda_i}\epsilon(\bj)\\
	&=\sum_{\bj\in \mcC}\prod_{i\in I,i\neq j_1}(-q^{j_1})^{\Lambda_{j_1}}(q^{j_1}-q^{i}-q^{j_1}y_1)^{\Lambda_i}y_1^{\Lambda_{j_1}}\epsilon(\bj)
	\end{align*} is in $\langle y_1^{\Lambda_{j_1}}\epsilon(\bj)\ | \ \bj\in \mcC\rangle$. By Theorem (\ref{thm:basis semi-ration algebra}) (1), $\prod_{i\in I, i\neq j_1}[(X_1-q^i)^{\Lambda_i}]^{-1}\epsilon(\bj)$ is in $\tilde{\mathcal{L}}$, thus \begin{align*}y_1^{\Lambda_{j_1}}\epsilon(\bj)&=-q^{-j_{1}}(X_1-q^{j_1})^{\Lambda_{j_1}}\epsilon(\bj)\\
	&=-q^{-j_{1}}\prod_{i\in I}(X_1-q^i)^{\Lambda_i}\prod_{i\in I, i\neq j_1}[(X_1-q^i)^{\Lambda_i}]^{-1}\epsilon(\bj)
	\end{align*}
	is in $\langle \prod_{i\in I}(X_1-q^i)^{\Lambda_i}\rangle$. So $\langle y_1^{\Lambda_{i_1}}\epsilon(\bi)\ | \ \bi\in \mcC\rangle =\langle \prod_{i\in I}(X_1-q^i)^{\Lambda_i}\rangle$ in $\tilde{\mathcal{L}}$, and then we are done.
\end{proof}

\begin{remark}  Denote by $$\tilde{\mathcal{L}}(\Lambda):=\tilde{\mathcal{L}}/\langle y_1^{\Lambda_{i_1}}\epsilon(\bi) | \bi\in \mcC \rangle.$$ 
	Similar to the proof of Lemma \ref{lem:nilp of y}, the elements $y_r$ are nilpotent in $\tilde{\mathcal{L}}(\Lambda)$. Therefore the elements $\prod_{i\in I}(X_r-i)$ (resp., $\prod_{i\in I}(X_r-q^i)$) are also nilpotent in $\tilde{\mathcal{L}}(\Lambda)$. 
\end{remark}

 The following result is important for us.

\begin{lemma} \label{lem:coverCRKLRA} We have $\Bbbk$-algebra isomorphsim $\mcR(\Lambda)\cong \tilde{\mathcal{L}}(\Lambda)$.
\end{lemma}

\begin{proof}  First note that if $f(y)\in \Bbbk[y]$ with $f(0)\neq0$,  the polynomial $f(y)-f(0)$ is nilpotent in $\tilde{\mathcal{L}}(\Lambda)$ by the nilpotency of $y_r$'s. So there exists some $g(y)\in \Bbbk[y]$ and $m\in \mathbb{N}$ such that $g(y)^m=0$ and $f^{-1}(y)=f(0)^{-1}\sum_{l=0}^m g(y)^l$ in $\mcR(\Lambda)$. Thus the  homomorphism $\mcR\hookrightarrow \tilde{\mathcal{L}}\twoheadrightarrow \tilde{\mathcal{L}}(\Lambda)$ is surjective and induces a surjective homomorphism $\pi\colon \mcR(\Lambda)\to \tilde{\mathcal{L}}(\Lambda)$. Let $\tilde{F}$ be the localization of the commutative ring $\Bbbk[y]^{S_n}$ of $S_n$-invariants in $\Bbbk[y]$ with respect to $\{f\in \Bbbk[y]^{S_n} \  | \ f(0)\neq 0\}$. Similar to the proof of (\ref{K(X)}), we have $\tilde{\mathcal{L}}=\mathcal{R}\otimes_{\Bbbk[y]^{S_n}}\tilde{F}$ by Theorem \ref{thm:basis semi-ration algebra}. Since the elements $y_r$ are nilpotent in
$\mathcal{R}(\Lambda)$, similar to the proof above, we know that if $f(y)\in K[y]^{S_n}$ with $f(0)\neq 0$, then it is a unit in $\mathcal{R}(\Lambda)$. Thus the homomorphism $K[y]^{S_n}\hookrightarrow \mathcal{R}\stackrel{\pi_1}\twoheadrightarrow \tilde{\mathcal{R}}(\Lambda)$ induces a morphism $\pi_2\colon \tilde{F}\to \mathcal{R}(\Lambda)$. Therefore we have an induced algebra homomorphism $\pi_1\otimes \pi_2\colon \tilde{\mathcal{L}}\to \mathcal{R}(\Lambda)$. The homomorphism $\pi_1\otimes\pi_2$ induces an algebra homomorphism $\pi'\colon \tilde{\mathcal{L}}(\Lambda)\to \mathcal{R}(\Lambda)$. It is easy to check that $\pi$ and $\pi'$ are two-sided inverses. So $\mcR(\Lambda)\cong \tilde{\mathcal{L}}(\Lambda)$.
\end{proof}

\section{The Brundan-Kleshchev isomorphism}

\subsection{Cyclotomic degenerate affine Hecke algebras}

Recall that the {\itshape cyclotomic degenerate affine Hecke algebra} is defined as
$$\mcH(\Lambda):=\mcH/\langle\prod_{i\in I}(X_1-i)^{\Lambda_i}\rangle.$$
By \cite[Subsection 3.1]{Brundan-Kleshchev09}, there is a system $\{e(\bi) \ | \ \bi\in \mcC\}$ of mutually orthogonal idempotents in $\mcH(\Lambda)$ such that $1=\sum_{\bi\in I^n}e(\bi)$ and $$e(\bi)\mcH(\Lambda)=\{h\in \mcH(\Lambda) \ | \ (X_r-i_r)^mh=0 \ \mbox{for all} \ 1\leq r\leq n \ \mbox{and} \ m\gg0\}.$$
It is easy to check that $X_re(\bi)=e(\bi)X_r$ for all $1\leq r\leq n$ and $\bi\in \mcC$, and the element $f(X)e(\bi)$ with $f(X)\in \Bbbk[X]$ is a unit in $e(\bi)\mcH(\Lambda)$ if and only if $f(\bi)\neq 0$. In this case, we write $f(X)^{-1}e(\bi)$ for the inverse. 
\begin{lemma} \label{lem:degenerate} For $1\leq r<n$ and $\bi\in I^{n}$, we have that
	\begin{equation} \label{deTre-eTr} T_re(\bi)=\begin{cases}
	e(\bi)T_r & \text{if $i_{r+1}=i_r$},\\
	e(\sigma_r(\bi))T_r+
	(X_r-X_{r+1})^{-1}(e(\sigma_r(\bi))-e(\bi))  & \text{if $i_{r+1}\neq i_r$}.
	\end{cases}
	\end{equation}
\end{lemma}

\begin{proof}   For any $1\leq s\leq n$, the element $$(\sigma_r(X_s)-\sigma(\bi)_s)e(\bi)=\begin{cases}
	(X_s-i_s)e(\bi)&\text{if $s\neq r,r+1$,}\\
	(X_{r+1}-i_{r+1})e(\bi)&\text{if $s=r,$}\\
	(X_r-i_r)e(\bi)&\text{if $s=r+1.$}
	\end{cases}$$ is nilpotent in $\mcH(\Lambda)e(\bi)$. Similarly, we can show that 
	$\parti_r((X_s-i_s)^m)e(\bi)=0$  by the binomial theorem for an integer $m\gg0$ whenever $i_r=i_{r+1}$. Therefore, if $i_r=i_{r+1}$, we get from (\ref{fT}) that
	\begin{align*}(X_s-i_s)^mT_re(\bi)=T_r(\sigma_r(X_s)-i_s)^me(\bi)+{\parti}_r((X_s-i_s)^m)e(\bi)=0\end{align*}
	whenever $m\gg0$. Thus $T_re(\bi)\in e(\bi)\mcH(\Lambda)$ and then $T_re(\bi)=e(\bi)T_re(\bi)=e(\bi)T_r.$ 
	
	If $i_{r+1}\neq i_r$, as a result of (\ref{fT}), it holds that
	\begin{align*}(X_s-\sigma_r(\bi)_s)^m[T_r(X_r-X_{r+1})+1]e(\bi)&=[T_r(X_r-X_{r+1})+1](\sigma_r(X_s)-\sigma_r(\bi)_s)^me(\bi)\\
	&=0\end{align*}
	whenever $m\gg0$. Therefore
	\begin{align*}T_r(X_r-X_{r+1})e(\bi)+e(\bi)=e(\sigma_r(\bi))[T_r(X_r-X_{r+1})+1]
	\end{align*}
	Then right-multiplying by $(X_r-X_{r+1})^{-1}e(\bi)$, we obtain $$T_re(\bi)=e(\sigma_r(\bi))T_re(\bi)-
	(X_r-X_{r+1})^{-1}e(\bi).$$
	Similarly, we can prove that $$e(\sigma_r(\bi))T_r=e(\sigma_r(\bi))T_re(\bi)-
	(X_r-X_{r+1})^{-1}e(\sigma_r(\bi)).$$
	Therefore $$T_re(\bi)=e(\sigma_r(\bi))T_r+
	(X_r-X_{r+1})^{-1}(e(\sigma_r(\bi))-e(\bi)).$$
\end{proof}

\subsection{Cyclotomic non-degenerate affine Hecke algebras}
Recall that the {\itshape non-degenerate cyclotomic affine Hecke algebra $\mcH_{q}(\Lambda)$} is
$$\mcH_{q}(\Lambda):=\mcH_{q}/\langle\prod_{i\in I}(X_1-q^i)^{\Lambda_i}\rangle.$$
Similar to \cite[Subsection 4.1]{Brundan-Kleshchev09}, there is a system $\{e(\bi) \ | \ \bi\in \mcC\}$ of mutually orthogonal idempotents in $\mcH_q(\Lambda)$ such that $1=\sum_{\bi\in I^n}e(\bi)$ and \begin{align*}e(\bi)\mcH_{q}(\Lambda)=\{h\in \mcH_{q}(\Lambda) \ | \ (X_r-q^{i_r})^mh=0 \ \mbox{for all} \ 1\leq r\leq n \ \mbox{and} \ m\gg0\}\end{align*}.
Accordingly, it holds that $X_re(\bi)=e(\bi)X_r$ for all $1\leq r\leq n$ and $\bi\in \mcC$, and the element 
$f(X)e(\bi)$ with $f(X)\in \Bbbk[X]$ is a unit in $e(\bi)\mcH_{q}(\Lambda)$ if and only if $f(q^{\bi})\neq 0$. We write $f(X)^{-1}e(\bi)$ for the inverse in this case.

Similar to the proof of Lemma \ref{lem:degenerate}, we can show that
\begin{lemma} For $1\leq r<n$ and $\bi\in I^{n}$, we have
	\begin{equation} \label{Tre-eTr} T_re(\bi)=\begin{cases}
	e(\bi)T_r & \text{if $i_r=i_{r+1}$},\\
	e(\sigma_r(\bi))T_r+\tfrac{(1-q)X_{r+1}}{X_{r+1}-X_r}(e(\sigma_r(\bi))-e(\bi)) & \text{if $i_r\neq i_{r+1}$}.
	\end{cases}
	\end{equation}
\end{lemma}

Let $$e(\mcC):=\sum_{\bi\in \mcC}e(\bi).$$
Then $e(\mcC)$ is a primitive central idempotent by (\ref{deTre-eTr}) and (\ref{Tre-eTr}).

\subsection{The BK isomorphism}  By Lemma \ref{lem:ccover}, Theorem (\ref{thm:basis semi-ration algebra}) and \ref{lem:degenerate}, (\ref{Tre-eTr})there is a homomorphism 
$$\rho \colon \tilde{\mathcal{L}}(\Lambda) \to \mcH(\Lambda)e(\mcC) (\mbox{resp}.\ \mcH_q(\Lambda)e(\mcC))$$ sending the generators $X_r, T_r$ to the same named elements, and $f^{-1}(X)\epsilon(\bi)$ with $f(\bi)\neq 0$ (resp. $f^{-1}(X)\epsilon(\bi)$ with $f(q^\bi)\neq 0$) to $f^{-1}(X)e(\bi)$. We have the following important result.
\begin{lemma} \label{lem:rho} $\rho$ is an algebra isomorphism.
\end{lemma}

\begin{proof} We only prove the claim for degenerate case since the proof of non-degenerate case is similar. Apparently, $\rho$ is surjective, so we only need to construct a left-inverse of $\rho$. By Lemma \ref{lem:ccover}, there is a homomorphism
	$$\tau \colon \mcH(\Lambda) \to \tilde{\mathcal{L}}(\Lambda)$$
	sending the generators $X_r,T_r$ to the same named elements. Let $\bi\in \mcC $ and $\bj\in I^n$. If $\bi\neq \bj$, then there is some $1\leq r\leq n$ such that $j_r\neq i_r$. We claim that $\epsilon(\bi)\tau(e(\bj))=0$. In fact, by the construction of $e(\bj)$, there is an integer $m\gg0$ such that $(X_r-j_r)^me(\bj)=0$. Thus we have $$(X_r-j_r)^m\epsilon(\bi)\tau(e(\bj))=\epsilon(\bi)\tau((X_r-j_r)^me(\bj))=0.$$
	The assumption $j_r\neq i_r$ implies that the element $(X_r-j_r)^{-1}\epsilon(\bi)\in \tilde{\mathcal{L}}(\Lambda)$. Thus we get 
	$$\epsilon(\bi)\tau(e(\bj))=(X_r-j_r)^{-m}(X_r-j_r)^m\epsilon(\bi)\tau(e(\bj))=(X_r-j_r)^{-m}\epsilon(\bi)0=0.$$
	Therefore, if $\bj\in I^n\setminus \mcC$ we have $\tau(e(\bj))=\sum_{\bi\in \mcC}\epsilon(\bi)\tau(e(\bj))=0.$ Therefore, if $\bj\in \mcC$, we obtain
	$$\tau(e(\bj))=\sum_{\bi\in \mcC}\epsilon(\bi)\tau(e(\bj))=\epsilon(\bj)\tau(e(\bj))=\epsilon(\bj)\sum_{\bi\in I^n}\tau(e(\bi))=\epsilon(\bj)\tau(1)=\epsilon(\bj).$$
	These show that $\tau|_{ \mcH(\Lambda)e(\mcC)}\colon \mcH(\Lambda)e(\mcC)\to \tilde{\mathcal{L}}(\Lambda)$ is an algebra homomorphism. It is easy to check that $\tau \rho$ is the identity on each generator of $\tilde{\mathcal{L}}(\Lambda)$. Thus $\rho$ is an isomorphism.
\end{proof}

By Lemma \ref{lem:coverCRKLRA} and Lemma \ref{lem:rho}, we have arrive at our main result.
\begin{theorem}\label{thm:mainresult}
 The cyclotomic KLR algebra of type A and blocks of the cyclotomic $($ degenerate $)$ affine Hecke algebra of a symmetric group can be realized as the same quotient of an algebra.
\end{theorem}

\end{document}